\documentclass[10pt,a4paper]{article}
\usepackage[utf8]{inputenc}
\usepackage{amsfonts,amsmath,amsthm,amscd,amssymb,latexsym}
\usepackage[margin=1.2in]{geometry}
\usepackage{tikz}
\usepackage{tabularx,ragged2e,booktabs,caption}
\usepackage{lscape} 
\usepackage{enumerate}               
\usepackage{titlesec}
\usepackage{mathtools}
\usepackage{pifont}
\usepackage{color}
\usepackage{makeidx}				
\usepackage{appendix} 
\usepackage{longtable}
\usepackage{tocloft}

\DeclareMathOperator{\End}{End}

\DeclareMathOperator{\LHC}{\text{LHC}}

\newcommand{\Z}{{\mathbb{Z}}}

\newcommand{\sq}{\mathbin{\square}}

\theoremstyle{plain}
\newtheorem{theorem}{Theorem}[section] 
\newtheorem{lemma}[theorem]{Lemma}
\newtheorem{proposition}[theorem]{Proposition}
\newtheorem{corollary}[theorem]{Corollary}

\theoremstyle{definition}
\newtheorem{example}[theorem]{Example}
\newtheorem{remark}[theorem]{Remark}

\newtheorem{problem}[theorem]{Problem}

\begin{document}

\title{Endomorphisms of The Hamming Graph and Related Graphs}
\author{Artur Schaefer\\
  {\small Mathematical Institute, University of St Andrews}\\
  {\small North Haugh, St Andrews KY16 9SS, UK}\\
  {\small  as305@st-andrews.ac.uk}}
\date{}
\maketitle

\begin{abstract}
 In this paper we determine all singular endomorphisms of the Hamming graph and other related graphs. The Hamming graph has vertices $\mathbb{Z}^{m}_n$ where two vertices are adjacent, if their Hamming distance is $1$. We show that its singular endomorphisms are uniform (each kernel has the same size) and that they are induced by Latin hypercubes (which essentially determines the number of singular endomorphisms). However, we do the same for its complement and some related graphs where the Hamming distance is allowed to be one of $1,...,k$, for some $1\leq k\leq m-1$. Ultimately, we consider the same situation where the vertices are tuples in $\mathbb{Z}_{n_1}\times\mathbb{Z}_{n_2}\times\cdots \times\mathbb{Z}_{n_m}$ (not all $n_i$ are equal).
\end{abstract}

\section{Introduction}\label{section1}

Recently, the search for graph endomorphisms has been motivated by not only by the mathematical field graph theory itself, but also from a different one, namely synchronization theory. The recent work on synchronization theory or rather synchronizing permutation groups \cite{pjc13,araujo15,araujo13} is considering transformation semigroups $S$ which contain a non-trivial permutation group $G$ such that $S=\langle G,T\rangle$, for a set $T$ of singular transformations (or maps). The main problem in synchronization theory is to classify all groups $G$ which are synchronizing, that is where $\langle G,t\rangle$ contains a transformation of rank $1$ (size of its image). A secondary problem is to find all tuples $(G,t)$ such that this semigroup is synchronizing. 

However, if a group is not synchronizing, then there is a map $t$ of minimal rank which is not synchronized by $G$. Maps of minimal ranks correspond to section-regular partitions, as shown by Neumann \cite{neumann}. Moreover, such partitions are uniform and pose an interesting combinatorial object by themselves; however, in \cite{pjc08} Cameron and Kazanidis introduced a graph theoretical approach towards this problem. They showed that if there is a map not synchronized, then there is a graph having this map as an endomorphism. The precise result is given by the next theorem.

\begin{theorem}[Thm 2.4 \cite{pjc08}]\label{thm1}
 A group $G$ does not synchronize a transformation $f$, if and only if there is non-trivial graph $X$ with complete core such that $\langle G,f\rangle\leq \End(X)$. 
\end{theorem}

Anyway, a synchronizing group is necessarily primitive \cite{pjclectureonsynchronization}, but there are primitive groups that are not synchronizing. Every non-synchronizing primitive group fails to synchronize at least one uniform transformation (that is, a transformation whose kernel has parts of equal size), and it has previously been conjectured (cf. \cite{araujo15}) that this was essentially the only way in which a primitive group could fail to be synchronizing – in other words, that a primitive group synchronizes every non-uniform transformation.

The first place to look for a counter-example is the so-called non-basic primitive groups; such a group is contained in the automorphism group of the Hamming graph. This led in part to the current research. Indeed, an infinite family of non-basic counter-examples was subsequently found, see \cite{araujo15}.


Using this Theorem \ref{thm1}, the study of synchronizing groups translates into the study of graph endomorphisms, and many (primitive) groups were shown to be non-synchronizing that way. Moreover, this theorem reignited the study of graph endomorphisms (which was possibly motivated by their application to synchronization theory \cite{godsil11,huang14,huang15}). 

In particular, in \cite{huang14} and \cite{huang15} the authors described endomorphisms of the Grassman graphs and graphs from alternating forms. These graphs are so-called distance-transitive graphs, and a survey on the classification of these graphs can be found in \cite{vanbon07}. Furthermore, the Hamming graph is a distance-transitive graph; so basically, the current research continues the analysis of singular endomorphisms of distance-transitive graphs

The Hamming graph $H(m,n)$ is the graph whose vertices are elements from $\mathbb{Z}_n^{m}$ where two vertices are adjacent, if their Hamming distance is exactly $1$. This graph is very famous and much is known about it, for instance this graph is actually the Cartesian product of $m$ complete graphs $K_n$, that is
\[K_n\sq \cdots \sq K_n.\]
Moreover, if $m=2$, then $H(m,n)$ is the strongly regular square lattice graph $L_2(n)$ (which in fact is an orthogonal array graph $OA(2,n)$ \cite[Thm.~ 10.4.2]{godsilbook}). This graph is one of the few families of strongly regular graphs whose minimum eigenvalue is $-2$ (cf. Seidel's theorem \cite{pjcdesignsbook}). In general, the connection between Hamming graphs and  coding theory is of major importance. 

In general, if $S$ is a subset of $\{1,...,n\}$, then we define a Hamming graph $H(m,n,S)$ to be a graph over the same vertex set as $H(m,n)$ whose vertices are adjacent, if their Hamming distance is in $S$. (For the reader who is familiar with association schemes: the graph $H(m,n,S)$ is a union of associates in the Hamming association scheme.) If $S$ consists of a single element $k$, then we write $H(m,n,k)$, and if $k=1$, then this is the Hamming graph $H(m,n)$.

Before moving on, we need to define $k$-dimensional layers (or $k$-layers) and systems of $k$-layers. From school everyone knows that one can draw a cube by drawing its layers iteratively. That is, a cube is a collection of two dimensional layers, which are squares. This concept applies to higher dimensions and is described here. First, a $k$-layer is the maximal clique of the Hamming graph $H(m,n,S)$, where $S=\{1,...,k\}$. Essentially, this is a set of the form
\[L=a+\langle e_{i_1},...,e_{i_k}\rangle,\]
where $a\in \mathbb{Z}_n^{m}$ and standard tuples $e_i=(0,...,0,1,0,...,0)$ with $1$ at the $i$th position. A system of $k$-layers is a set of parallel $k$-layers which cover $\mathbb{Z}_n^{m}$. 

The $k$-dimensional layers play an important role for Hamming graphs, and so does their number; let $h_k(m,n)$ denote this number. The maximal cliques in the square lattice graph are $1$-layers their number is $2n$. Also there are $mn$ layers of dimension $m-1$ in $H(m,n)$. To obtain a $k$-layer, we need to choose $k$ of the $m$ coordinates, and for each such choice there are $m-k$ coordinates which are fixed. So, the value for $h_k(m,n)$ is
\begin{equation*}
 h_{k}(m,n)=\binom{m}{k} n^{m-k}.
\end{equation*}
Applying this formula to the number of maximal cliques in $H(m,n)$, which in fact are $1$-layers, reveals that there are $h_1(m,n)=mn^{m-1}$ of them. Similarly, the number of $(m-1)$-layers is $h_{m-1}(m,n)=mn$.

However, we are still missing the cuboidal versions. The cuboidal Hamming graphs has vertices given by the set $\mathbb{Z}_{n_1}\times \mathbb{Z}_{n_2}\times \cdots \times \mathbb{Z}_{n_m}$ (with possibly distinct $n_i$) where two vertices are adjacent, if their Hamming distance is in a set $S\subseteq \{1,...,m\}$. These graphs are denoted by $H(n_1,n_2,...,n_m,S)$ and we (usually) assume $n_1\geq n_2\geq \cdots \geq n_m$. (For convenience we write $H(n_1,n_2,...,n_m)$, if $S=\{1\}$.) Clearly, $k$-layers are defined respectively.

Next, we need to introduce Latin hypercubes of class $k$. These cubes have only occurred a few times in the literature. They initially appeared 1950 in \cite{kishen50}, but were rediscovered in the 70's by Keedwell and Denes \cite{denes74} and rather recently by Ethier \cite{ethier08}, Moura et. al. \cite{moura15}, and Schaefer \cite{schaeferlatinexistence,schaeferlatinembeddings,schaeferlatinextension}. A $d$-dimensional Latin hypercube of order $n^{k}$ and class $k$ is an $n\times n\times \cdots \times n$ ($d$ times) array based on $n^{k}$ distinct symbols, each repeated $n^{d-k}$ times, such that each occurs exactly once in each $k$-layer. We will write $\LHC(d,n,k)$ for such cubes, and $\LHC(d,n)$, if $k=1$. The number of all Latin hypercubes $\LHC(d,n,k)$ is denoted by $\# \LHC(d,n,k)$.

The corresponding cuboidal versions are somewhat more complicated. Let $n_1\geq n_2\geq \cdots \geq n_d\geq 2$ be integers. A Latin hypercuboid of dimension $d$, type $(n_1,...n_d)$, order $n$ and class $k$ is an $n_1\times n_2 \times \cdots \times n_d$ array based on $n=\prod\limits_{i=1}^{k} n_i$ distinct symbols, such that in every $k$-layer with $n$ entries each symbol occurs exactly once and in any other $k$-layer with less entries each symbol occurs at most once. We write $\LHC(n_1,...,n_d,k)$ for such an hypercuboid. Again, more on these objects can be found in the research of Schaefer.

The most important results in this paper are the following ones. Regarding the kernel structure of the endomorphisms of $H(m,n,S)$, we obtain the next theorem.

\begin{theorem}\label{thm2}
 Let $\Gamma$ be the Hamming graph $H(m,n)$ and $\Delta$ be the graph $H(m,n,m)$. Then it holds:
 \begin{enumerate}
  \item A singular endomorphism of $\Gamma$ is uniform of rank $n^{k}$, for any $1\leq k\leq m-1$, and its image is a $k$-layer.
  \item A singular endomorphism of $\Delta$ is uniform of rank $n^{k}$, for any $1\leq k\leq m-1$.
  \item A singular endomorphism of $\overline{\Gamma}$ is a colouring of rank $n^{m-1}$ (hence uniform). 
  \item A singular endomorphism of $\overline{\Delta}$ is a colouring of rank $n^{m-1}$ (hence uniform).
 \end{enumerate}
\end{theorem}

The next result constitutes that the singular endomorphisms of $H(m,n)$ are induced by Latin hypercubes.

\begin{theorem}\label{thm3}
 The number of singular endomorphisms of $H(m,n)$ of rank $n^k$, for $1\leq k\leq m-1$, is given by the formula
\begin{equation*}
\binom{m}{k}\cdot n^{m-k}\cdot k!\cdot \left( \sum\limits_{\substack{P \text{ partition of } \{1,...,m\} \\ \text{with } k \text{ parts}}} \prod\limits_{X \in P} \#\LHC(|X|,n) \right),
\end{equation*}
where the product runs over all parts in $P$; $|X|$ is the size of the part $X\in P$ and $\#\LHC(d,n)$ is the number of Latin hypercubes of dimension $d$ of order $n$ (and class $1$).
\end{theorem}

Similarly, we obtain the corresponding results for other Hamming distances and hypercuboids.

\begin{theorem}\label{thm4}
 Let $S=\{1,...,k\}$. The singular endomorphisms of $H(m,n,S)$ are uniform and have rank $n^{d}$ with image a $d$-layer, for some $k \leq d \leq m-1$. Moreover, if they are of minimal rank, then they are induced by Latin hypercubes of class $k$.
\end{theorem}

\begin{theorem}\label{thm5}
\begin{enumerate}
 \item The singular endomorphisms of $H(n_1,...,n_m)$ are uniform of rank $n_1\cdot \prod\limits_{i\in I} n_i$, where $I$ is a proper subset of $ \{n_2,...,n_m\}$.
\item The singular endomorphisms of $H(n_1,...,n_m,S)$, for $S=\{1,...,k\}$, of minimal rank $n_1\cdots n_k$ are Latin hypercuboids of class $k$.
 \end{enumerate}
\end{theorem}

\section{The Hamming Graph and its Complement}\label{section2}

\subsection{Endomorphisms of the Hamming Graph $H(m,n)$}

In dimension $m=2$ the Hamming graph $H(2,n)$ is identified as either the square lattice graph $L_2(n)$ or the orthogonal array graph $OA(2,n)$. However, in \cite{godsil11} Godsil and Royle determined the singular endomorphisms of this graph. They showed that the singular endomorphisms are $n$-colourings; however, $n$-colourings are Latin squares by \cite[Thm.~ 10.4.5]{godsilbook}.

\begin{proposition}
 The number of proper endomorphisms of $L_2(n)$ is
\begin{equation}
\# \text{ maximal cliques } \cdot \# \text{ Latin squares of order } n.
\end{equation}
Note that $L_2(n)$ has $2n$ distinct maximal cliques.
\end{proposition}




To show that the singular endomorphisms are uniform in higher dimensions we essentially follow the same strategy in the proof. The first part of Theorem \ref{thm2} is the following.

\begin{theorem}\label{theoremendomorphismsofH(m,n)}
A singular endomorphism of $H(m,n)$ is uniform of rank $n^k$, for some $1\leq k\leq m-1$, and its image is a layer of dimension $k$.
\end{theorem}
%
%

This theorem is a consequence of the following lemma.

\begin{lemma}\label{lemmaonhypercubes}
 Let $\phi$ be a singular endomorphism of $H(m,n)$, and let $l$ be a $k$-layer. Then $\phi(l)$ is a layer of dimension $d$, where $1\leq d \leq k$.
\end{lemma}
\begin{proof}
 We will use induction on $m$ and $k$. Let $A(m,k)$ be the hypothesis. The hypothesis is satisfied for the following initial values $A(2,1),A(2,2)$ and $A(m,1)$ are true. Assume the hypothesis holds for $A(m,k)$ and show it holds for $A(m,k+1)$. 
 
 Let $l$ be a $(k+1)$-layer. Then, we can split $l$ into parallel $k$-layers $l_1,...,l_n$. By induction $\phi(l_i)$ is a $k$-layer or a layer of smaller dimension, for all $i$. Now, if the dimensions of, say, $\phi(l_1)$ and $\phi(l_2)$ would differ, then there would be two maximal cliques (lines) $c_1$ and $c_2$ connecting $l_1$ and $l_2$ such that at least one of $\phi(c_1)$ and $\phi(c_2)$ would not be a line in the image of $\phi$ (cf. Figure \ref{figureforH(m,n)}). A contradiction. Therefore, all $\phi(l_i)$ have the same dimension, say $d$.
 
 Using the same argument, we see that each $l_i$ is collapsed to $\phi(l_i)$, and that the layers $\phi(l_i)$ must form a $(d+1)$-layer. Thus, the image $\phi(l)$ is a $(d+1)$-layer. Note, each $l_i$ is collapsed to $\phi(l_i)$ uniformly; otherwise, by essentially the same argument we would be able to find a maximal clique which is not mapped to a maximal clique.
\end{proof}

\begin{figure}[t!]
\begin{center}

\begin{tikzpicture}

\tikzstyle{every node}=[scale=.6]
\draw (0,0) ellipse (1cm and 2cm) node[scale=1,sloped,fill=white] {$l_1$};;
\draw (2.5,0) ellipse (1cm and 2cm) node[scale=1,sloped,fill=white] {$l_2$};

\draw (7.5,-0.1) ellipse (1cm and 1.3cm) node[scale=1,sloped,fill=white] {$\phi(l_1)$};;
\draw (10,0) ellipse (1cm and 2cm) node[scale=1,sloped,fill=white] {$\phi(l_2)$};;
\node (9) at (1.25,1.25) {$c_1$};
\node (10) at (1.25,-1.25) {$c_2$};
\draw[->] (4,0) -- (6,0) node[pos=0.5,sloped,above,scale=1.5,fill=white] {$ \phi $};
\node (16) at (10,1.35) {$\phi(c_1)$};
\node (16) at (8.65,-1.35) {$\phi(c_2)$};

\tikzstyle{every node}=[scale=.5,circle,fill=black]
\node (1) at (0,1) {};
\node (2) at (0,-1) {};
\node (3) at (2.5,1) {};
\node (4) at (2.5,-1) {};
\node (5) at (-1.5,1) {};
\node (6) at (4,1) {};
\node (7) at (-1.5,-1) {};
\node (8) at (4,-1) {};

\draw (1)--(3);
\draw (2)--(4);
\draw (1)--(5);
\draw (6)--(3);
\draw (2)--(7);
\draw (4)--(8);

\node (11) at (7.5,-1) {};
\node (12) at (10,1) {};
\node (13) at (10,-1) {};
\node (14) at (6,-1) {};
\node (15) at (11.5,-1) {};
\node (16) at (11.5,1) {};

\draw[thick] (11)--(14);
\draw (11)--(12);
\draw (11)--(13);
\draw (13)--(15);
\draw (12)--(16);

\end{tikzpicture}
\caption{Impossible configuration}
\label{figureforH(m,n)}
\end{center}
\end{figure}
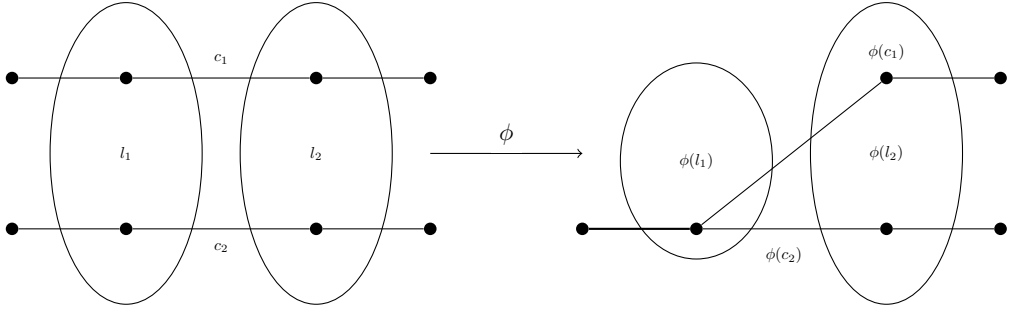
%
%
%

\begin{proof}[Proof of Thm.~ \ref{theoremendomorphismsofH(m,n)}]
Let $\phi$ be a singular endomorphism and let $l$ be the whole $m$-layer. By the previous lemma $\phi(l)$ is a $k$-layer where $1\leq k<m$.
\end{proof}
\begin{corollary}
 For any singular endomorphism $\phi$ there is a maximal number $k$, such that $\phi$ maps $k$-dimensional layers to $1$-dimensional layers.
\end{corollary}

The following should be also clear.

\begin{lemma}\label{lemmapreimageisalatinhypercube}
 If a singular endomorphism $\phi$ of $H(m,n)$ collapses a $k$-dimensional layer $l$ to a line, then $\phi^{-1}(l)$ is a Latin hypercube.
\end{lemma}
%
\subsection{The Complement of the Hamming Graph}

The complement of the Hamming graph $H(m,n)$ is the graph $H(m,n,S)$, where $S=\{2,...,m\}$, and two vertices are adjacent if their Hamming distance is not $1$. Again, we cover the $2$-dimensional case first. Here, the maximal cliques are given by 
\begin{equation*}
\lbrace (i,gi): i=1,...,n \rbrace \quad \text{for } g\in S_n. 
\end{equation*}

A recent result by David Roberson says that the singular endomorphisms of strongly regular graphs are colourings. Because the square lattice graph is a strongly regular, its complement is strongly regular, as well. So, again by Theorem 10.4.5 in \cite{godsilbook}, its singular endomorphisms are extensions of the corresponding orthogonal array. However, this time we need to interpret these endomorphisms differently, that is they collapse either the rows or the columns. 

\begin{proposition}
 The singular endomorphisms of $\overline{L_2(n)}$ are collapsing rows or columns and their number is $|S_n\wr S_2|=2\cdot (n!)^2$.
\end{proposition}

\begin{example}
 Two endomorphisms are the following:
 \[
 \begin{pmatrix}
   2&2&2\\
   4&4&4\\
   9&9&9\\
 \end{pmatrix},\begin{pmatrix}
   4&8&3\\
   4&8&3\\
   4&8&3\\
   \end{pmatrix}.\]
\end{example}

Next, we move on to higher dimensions. Recall, a maximal clique in $\overline{H(2,n)}$ is of the form $\lbrace (gi,i): i=1,...,n \rbrace $ for a permutation $g\in S_n$, and when considering these as $1$-dimensional Latin rows, then the next result says that the maximal cliques of $\overline{H(m,n)}$ form Latin hypercubes (or orthogonal arrays, or MDS-codes).

\begin{lemma}
 The maximal cliques in $\overline{H(m,n)}$ are in $1-1$ correspondence with Latin hypercubes of dimension $m-1$ and order $n$ (and class $1$).
\end{lemma}
\begin{proof}
 First, we note that a Latin hypercube is a maximal clique of size $n^{m-1}$. Hence, the clique number is $n^{m-1}$. We are going to use induction on $m$. The case $m=2$ is clear. Let $C$ be a maximal clique in $\overline{H(m,n)}$. Pick a layer system $l_i$ of $(m-1)$-dimensional layers, for $i=1,...,n$. Each layer is a subgraph isomorphic to $\overline{H(m-1,n)}$, so it has clique number $n^{m-2}$. Moreover, each layer contains exactly $n^{m-2}$ points of $C$, since otherwise, if there would be one layer containing at least $n^{m-2}+1$ points of $C$, it would have a maximal clique of size $n^{m-2}+1$, contradicting the induction hypothesis. Therefore, the intersection $C\cap l_i$ is a maximal clique for $\overline{H(m-1,n)}$ and has $n^{m-2}$ points. Intersecting $C$ with all possible layers of dimension $m-1$, determines the coordinates of the points of $C$ and it turns out that $C$ is a Latin hypercube of dimension $m-1$.

\end{proof}

The second item in Theorem \ref{thm2} is the following.

\begin{theorem}
 The graph $\overline{H(m,n)}$ is a pseudo-core, i.e., all singular endomorphisms have rank $n^{m-1}$ and are uniform.
\end{theorem}
\begin{proof}
Let $c_1$ and $c_2$ be two maximal cliques which are collapsed by $\phi$, say, $\phi(c_1)=\phi(c_2)=c$. Since $c_1\neq c_2$, there are points $a\in c_1$ and $b\in c_2$ with $\phi(a)=\phi(b)$ and $a\neq b$; thus, $ a$ and $b$ are on a $1$-dimensional layer. Let $x$ be a point on a $1$-dimensional layer (line) through $a$ which does not contain $b$. Any point not in $c_1$ is non-adjacent to exactly $m$ points of $c_1$ and adjacent to the rest of them; so $x$ is non-adjacent to $m$ points in $c_1$ including the point $a$. For $x$ is adjacent to $b$, the point $\phi(x)$ is adjacent to $\phi(a)=\phi(b)$; therefore, $\phi(x)$ is adjacent to $m-1$ points of $\phi(c_1)=c$, and thus $\phi(x)$ is in $c$.

 Since $x$ is chosen arbitrarily on the $1$-dimensional layer, all the $1$-dimensional layers through $a$ not containing $b$ are mapped to $c$. Switching the roles of $a$ and $b$ with one of the new points mapped to $c$ and iterating this argument shows that all the points are mapped to $c$.
\end{proof}

\section{The Categorial Product of Complete Graphs and \\its Complement}\label{section3}

\subsection{The Categorial Product of Complete Graphs $H(m,n,m)$}

In this section another well-known graph product is considered, namely the categorial product of complete graphs. In particular, we are concerned with the product of $m$ copies of the complete graph $K_n$:
\[K_n\times \cdots \times K_n.\]
The vertices are $m$-tuples, and two vertices adjacent if they differ in every coordinate. So, by using the notation from above this graph is $H(m,n,m)$.

Again, we show that singular endomorphisms are uniform and of rank $n^{k}$, for $1\leq k\leq m-1$. But before, we need some auxiliary lemmata.

\begin{lemma}
For $m\geq 2$ and $n\geq 3$ the following is true.
 \begin{enumerate}
  \item The clique number and the chromatic number of $H(m,n,m) $ are equal to $n$.
  \item The maximal cliques are given by
 \begin{equation*}
  \lbrace (g_1i,g_2i,...,g_{m-1}i,i): i=1,..,n \rbrace,\quad \text{for } g_1,...,g_{m-1}\in S_n.
 \end{equation*}

  \item The number of maximal cliques in $H(m,n,m) $ is $(n!)^{m-1}$.
  \item The automorphism group of $H(m,n,m) $ acts transitively on the maximal cliques.
 \end{enumerate}

\end{lemma}
\begin{proof}
 Note, that the diagonal consisting of the points $(i,...,i)$, for $1\leq i\leq n$, is a clique of size $n$. Also, any layer in a layer system of $(m-1)$-layers contains a diagonal point. Thus, a map mapping an $(m-1)$-layer to the respective point is a singular endomorphism of rank $n$. Hence, it is a colouring. 

In $H(m,n,m)$ two vertices are adjacent if none of their coordinates are equal. So, take $g_1,...,g_{m-1}\in S_n$, then the set 
\begin{equation*}
\lbrace (g_1i,g_2i,...,g_{m-1}i,i): i=1,...,n\rbrace 
\end{equation*}
 forms a maximal clique. In fact, every combination of elements of $S_n$ provides a new clique and all maximal cliques are given this way. Their number is $(n!)^{m-1}$. The last result is obvious.
\end{proof}
%
%


\begin{lemma}\label{lemmacollapselayer}
 Suppose $\phi$ is a singular endomorphism of $H(m,n,m)$. Let $x_1$ and $x_2$ be two distinct points with $\phi(x_1)=\phi(x_2)$ and $l$ the minimal layer containing both points. Then, $l$ is mapped uniformly to $\phi(x_1)$.
\end{lemma}
\begin{proof}
Since $\phi$ is transitive on the maximal cliques, we may assume that $x_1=(1,...,1,1)$ and that $\phi(k,...,k)=(k,...,k)$, for all $k=1,...,n$. We use induction on the Hamming distance $d$ between $x_1$ and $x_2$. 
 
 Suppose $d=1$, then we can assume that $x_2=(1,...,1,2)$. The points $ y_k=(k,...,k,1) $ are adjacent to $x_2$ and adjacent to $(i,...,i)$, for $i\notin \{1,k\}$; therefore, $y_k$ is mapped to $(k,...,k)$. By the same argument it follows that the points $(1,...,1,i)$ are mapped to $\phi(x_1)$, for all $1\leq i\leq n$.
 
 Now, assume $d>1$ and that the hypothesis holds for smaller distances. Again, we may assume that $x_2=(1,...,1,a_1,...,a_d)$, where none of the $a_i$ is $1$. We show that it is sufficient to set $x_2=(1,...,1,2,...,2)$. As above, the points $y_k=(\underbrace{k,...,k}_{m-d},1,...,1)$ are adjacent to $x_2$ and, thus, they are mapped to $(k,...,k)$, for all $k\neq 1$. Hence, the point $x_2'=(\underbrace{1,...,1}_{m-d},2,...,2)$ is mapped to $\phi(x_1)$. So, set $x_2=x_2'$. 
 
 Next, since the point $y_k$, is mapped to $(k,...,k)$, the points $(\underbrace{a,...,a}_{m-d},b,...,b)$ are mapped to $(a,...,a)$, for all $a,b\in \Z_n$. Similarly, the points
 \begin{equation*}
  (\underbrace{a,...,a}_{m-d},b,...,b,b,c),
  (a,...,a,b,...,b,c,b),...,(a,...,a,c,b,...,b,b)
  \end{equation*} 
  and
  \begin{equation*}
   (a,...,a,\underbrace{b_1,...,b_d}_{\text{none of them }a})
  \end{equation*}
 are mapped to $(a,...,a)$, for all $a,b,c\in\Z_n$. By the induction hypothesis, the layers of dimension $\leq d-1$ inside $l$ are mapped to $\phi(x_1)$, but then it follows that all points in $l$ are mapped to $\phi(x_1)$. Uniformity is clear.
\end{proof}

The following covers the third item in Theorem \ref{thm2}. 

\begin{theorem}
 The singular endomorphisms of $H(m,n,m)$ are uniform of ranks $n^k$, for $1\leq k\leq m-1$.
 \end{theorem}
\begin{proof}
 Let $\phi$ be a singular endomorphism with $\phi(x_1)=\phi(x_2)$, for some distinct points $x_1$ and $x_2$. Thus, pick $x_1$ and $x_2$ with the maximal distance $d$ among all the points collapsed by $\phi$. Without loss of generality, $x_1=(1,...,1)$, $x_2=(1,...,1 \underbrace{2,...,2}_{d})$ and $\phi(i,...,i)=(i,...,i)$, for all $1\leq i\leq n$.
 
 By Lemma \ref{lemmacollapselayer}, the layer $l=\{(1,...,1,a_1,...,a_d):\ a_i \in \Z_n \}$ is mapped to $x_1$. Also, by the arguments used in Lemma \ref{lemmacollapselayer} the layers $l+\lambda(1,...,1)$ are mapped to the point $(\lambda,...,\lambda)$, for $1\leq \lambda\leq n$. So, pick another layer $\tilde{l}=\{(x_1,...,x_{m-d},a_1,...,a_d):\ a_1,...,a_d \in \Z_n \}$ for some $x_1,...,x_{m-d}\in \Z_n$. We show that there is a point $x$ with $\phi(\tilde{l})=x$. In other words, show that 
 \begin{equation*}
  \phi(x_1,...,x_{m-d},1,...,1)=\phi(x_1,...,x_{m-d},2,...,2).
 \end{equation*}
 
 Pick two maximal cliques $c_1$ and $c_2$ as follows. Let $c_1= \{y_1,y_2,y_3,...,y_n \}$ be a maximal clique, where $y_i=(j,...,j)$, for $j\in \Z_n$ and $i\geq 3$, and $y_1$ is an arbitrary point not mapped to $(j,...,j)$, for any $j$. It follows, that this determines the point $y_2$. On the other hand, let $c_2= \{z_1,z_2,z_3,...,z_n \}$ be a maximal clique with $z_1=y_1$ and $z_2=y_2+(\underbrace{0,...,0}_{m-d},1,...,1)$. Given this, we are able to choose the missing $z_i$ such that $z_i$ is mapped to $(j,...,j)$, where $j$ is determined by $y_i$, for $i\geq 3$. By construction, $c_1$ and $c_2$ are maximal cliques, and since $\phi(z_i)=\phi(y_i)$, for $i=1,3,4,...,n$, it holds $\phi(z_2)=\phi(y_2)$. However, the distance between $y_2$ and $z_2$ is $d$ and, thus, by Lemma \ref{lemmacollapselayer} the layer containing both points is mapped to a single point. Using different sets $c_1$ and $c_2$, we can show that all choices of $\tilde{l}$ are mapped to points.
\end{proof}

\subsection{The Complement Graph $\overline{H(m,n,m)}$}

As for $\overline{H(m,n)}$, we show that $\overline{H(m,n,m)}$ is a pseudo-core. Recall, two vertices in $\overline{H(m,n,m)}$ are adjacent, if their Hamming distance is in $\{1,...,m-1\}$. The following facts are obvious.

\begin{lemma}
 For $m\geq 2$ and $n\geq 3$ it holds:
\begin{enumerate}
  \item The maximal cliques are given by the $(m-1)$-dimensional layers.
  \item The number of maximal cliques in $\overline{H(m,n,m)} $ is $h_{m-1}(m,n)=mn$.
  \item The automorphism group of $\overline{H(m,n,m)} $ acts transitively on the maximal cliques. 
\end{enumerate}
\end{lemma}
%
%


By a straightforward combinatorial observation we obtain the last part of Theorem \ref{thm2}. 

\begin{theorem}
 The graph $\overline{H(m,n,m)} $ is a pseudo-core whose endomorphisms are uniform.
\end{theorem}
\begin{proof}
 Let $\phi$ be a singular endomorphism and assume $\phi$ maps the two maximal cliques $c_1$ and $c_2$ to $c$. Since the automorphism group is transitive on the maximal cliques, we may assume that $c=c_1$.

 We know that the maximal cliques are layers, so suppose $c_1$ and $c_2$ are parallel layers (with respect to the same coordinate). Pick a point $fx$ not in $c_1\cup c_2$ and let $l$ be an $(m-1)$-dimensional layer through $x$ not parallel to $c_1$. Then, 
\begin{equation*}
 |c_1\cap l|+|c_2\cap l|=n^{m-2}+n^{m-2}.
\end{equation*}
All these $2n^{m-2}$ points are pairwise adjacent. Also, they cannot be mapped to a single $(m-2)$-dimensional sublayer $\tilde{l}$ of $c$, since there are too few points in $\tilde{l}$. Hence, pick $m$ of the points which are in no $(m-2)$-dimensional layer. The image $\phi(x)$ has to be adjacent to all of the points, but the only points which are adjacent to all of the $m$ points are the points in $c$. Therefore, $\phi(x)$ is mapped to $c$.

On the other hand, suppose $c_1$ and $c_2$ are not parallel. Again, pick $x$ not in either of the cliques. Then, there is an $(m-1)$-dimensional layer $l$ intersecting with both $c_1$ and $c_2$, and it holds
\begin{equation*}
 |c_1\cap l|+|c_2\cap l|-|c_1\cap c_2\cap l|=n^{m-2}+n^{m-2}-n^{m-3}>n^{m-2}.
\end{equation*}
Again, these common points are pairwise adjacent and thus cannot be mapped to an $(m-2)$-dimensional layer. As in the last case, we can pick $m$ points in the image which are adjacent to $\phi(x)$ and which have $c$ as the only points adjacent to all of the $m$ points.

\end{proof}

\section{The Number of Endomorphisms and Latin hypercubes}\label{section4}

After determining the uniformity of the singular endomorphism in the previous sections, we are going to count the singular endomorphisms and prove Theorem \ref{thm3}. In detail, we derive formulae for the number of (singular) endomorphisms of the graphs $H(m,n), \overline{H(m,n)}$ and $\overline{H(m,n,m)}$. Unfortunately, further research is necessary to find the number of endomorphisms of $H(m,n,m)$. 

It is straightforward to construct singular endomorphisms for $\overline{H(m,n)}$ and $\overline{H(m,n,m)}$; so because their singular endomorphisms are colourings (Theorem \ref{thm2}), it is easy to find formulae. However, the Hamming graph $H(m,n)$ has endomorphisms of various ranks, namely, ranks $n^{k}$, for every $1\leq k\leq m-1$; so the formula will look somewhat more complicated. Recall, that the singular endomorphisms are Latin hypercubes, by Lemma \ref{lemmapreimageisalatinhypercube}; thus, the formula depends on the number of Latin hypercubes. (An very nice publication containing the number of Latin hypercubes is given by \cite{mckay08}.)



\begin{theorem}\label{theoremcountingendomsofhamminggraph}
 The number of singular endomorphisms of $H(m,n)$ of rank $n^k$, for $1\leq k\leq m$, is given by the formula

\begin{equation*}
\binom{m}{k}\cdot n^{m-k}\cdot k!\cdot \left( \sum\limits_{\substack{P \text{ partition of } \{1,...,m\} \\ \text{with } k \text{ parts}}} \prod\limits_{X \in P} \#\LHC(|X|,n) \right),
\end{equation*}
where the product runs over all parts in $P$; $|X|$ is the size of the part $X\in P$ and $\#\LHC(d,n)$ is the number of Latin hypercubes of dimension $d$ of order $n$ (and class $1$).
\end{theorem}
\begin{proof}
Let $\phi$ be a singular endomorphism. Since the image of $\phi$ is a $k$-dimensional layer (see Theorem \ref{theoremendomorphismsofH(m,n)}), we have $h_{k}(m,n)=\binom{m}{k}n^{m-k}$ choices to choose such a layer. We choose $k$ of the $m$ coordinates, say, $x_1,...,x_k$ which will determine the points of the image. Now, $\phi$ can be obviously described by a function onto the chosen $k$ coordinates:
\begin{equation*}
 \phi:\ (x_1,...,x_m) \mapsto (\phi_1(x_1,...,x_m),...,\phi_k(x_1,...,x_m),a_{k+1},...,a_m),
\end{equation*}
for some $a_{k+1},...,a_m\in \Z_n$. We show that each $\phi_i$ corresponds to a Latin hypercube.  

Let $x$ be a point in the image of $\phi$ and $e_1,...,e_m$ the standard basis of $\Z_n^m$. Consider the line $l:=x+\langle e_i \rangle$, for some $1\leq i \leq k$. The pre-image $\phi^{-1}(l)$ is determined by $\phi_i$; in addition, it is a Latin hypercube (by Lemma \ref{lemmapreimageisalatinhypercube}). Therefore, each of the functions $\phi_i$ are determined by Latin hypercubes. 

Next, suppose $\phi_1$ is given by a Latin hypercube of dimension $d$. It follows, that $\phi_1(x_1,...,x_m)=\phi_1(x_{i_1},...,x_{i_d})\overset{\text{Wlog}}{=}\phi_1(x_1,...,x_d)$, for $i_j\in \{1,...,m\}$. Assume there is another function, say, $\phi_2$ which depends on at least one of the coordinates $x_1,...,x_d$, say, $x_1$. In other words, assume that $\phi_1$ and $\phi_2$ depend on a common coordinate. Then, the line $x+\langle e_1 \rangle$ would be mapped to two distinct lines by $\phi_1$ and $\phi_2$, respectively. This is a contradiction, as a map cannot do that. Therefore, the $m$ coordinates are partitioned into $k$ parts. At last, each part has to be matched to a function $\phi_i$, for $i=1,...,k$; this provides $k!$ choices.
\end{proof}

We provide a small example to show how to use this formula.

\begin{example}
 Count the singular endomorphisms of $H(4,3)$. At first we need to partition the set $\{1,2,3,4\} $ into $1,2$ and $3$ parts with respect to the different values for $k$. (Note, the number of partitions of $\{1,\ldots,n\}$ into $k$ parts is the Stirling number of second kind, $S(n,k)$).
 \begin{center}
 \begin{tabular}{|c|c|}\hline
 $k$& Partitions \\\hline\hline
  $k=1$& $\{\{1,2,3,4\}\}$\\\hline
  $k=2$& $\{\{1\},\{2,3,4\}\},\{\{2\},\{1,3,4\}\},\{\{3\},\{1,2,4\}\},\{\{4\},\{1,2,3\}\}$\\
  &	 $\{\{1,2\},\{3,4\}\},\{\{1,3\},\{2,4\}\},\{\{1,4\},\{2,3\}\}$\\\hline
   $k=3$& $ \{\{1\},\{2\},\{3,4\}\},\{\{1\},\{2,3\},\{4\}\},\{\{1\},\{3\},\{2,4\}\}$\\
   &	$ \{\{1,2\},\{3\},\{4\}\},\{\{1,3\},\{2\},\{4\}\},\{\{1,4\},\{2\},\{3\}\}$\\\hline
 \end{tabular}
\end{center}
The number of Latin hypercubes is given in the next table.
\begin{center}
 \begin{tabular}{|c|c|}\hline
 $d$& $\#\LHC(d,3)$ \\\hline\hline
 $1$& $3!$\\\hline
 $2$& $3!\cdot 2$\\\hline
 $3$& $3!\cdot 2^2$\\\hline
 $4$& $3!\cdot 2^3$\\\hline 
 \end{tabular}
\end{center}
Eventually, we obtain for the different $k$:
\begin{align*}
 k=1: \#&= \binom{4}{1} \cdot 3^3 \cdot 1!\cdot \#\LHC(4,3) \\
	&=5184, \notag \\
 k=2: \# &=\binom{4}{2}\cdot 3^2 \cdot 2!  \cdot  \left(4 \cdot \#\LHC(1,3) \cdot \#\LHC(3,3) + 3 \cdot \#\LHC(2,3)^2 \right) \\
 &=108864, \notag \\
 k=3: \# &=\binom{4}{3} \cdot 3^1 \cdot 3! \cdot 6\cdot \#\LHC(1,3)^2 \cdot \#\LHC(2,3) \\
 &=186624. \notag
\end{align*}
Consequently, $H(4,3)$ admits $5184+108864+186624=300672$ singular endomorphisms.
\end{example}

\begin{corollary}
 The singular endomorphisms of $H(m,n)$ correspond to Latin hypercubes of class $1$ and dimension less than $m$. 
\end{corollary}

 Next, we turn to the graphs $\overline{H(m,n)}$ and $\overline{H(m,n,m)}$. In order to provide the number of singular endomorphisms, we need to define two combinatorial numbers. First, by $P_1(m,n)$ we denote the number of partitions of the hypercube $\mathbb{Z}_n^m$ into $1$-dimensional layers. (Alternatively, this number is the number of tilings of the $m$-dimensional cube with side $n$ with $n\times 1\times \cdots \times 1$ tiles.) We call it the \textit{'Jenga'-number}, due to the famous wooden building block game for children. This description is also equivalent to the partition of $\mathbb{Z}_n^m$ into non-intersecting maximal cliques of $H(m,n)$. In this regard, the number $P_2(m,n)$ denotes the number of partitions of $\mathbb{Z}_n^m$ into maximal cliques of $H(m,n,m)$. (See Table \ref{table1} and \ref{table2} for small values of $P_1(m,n)$ and $P_2(m,n)$).

\begin{example}
 Consider the points of $\Z_3^3$. We need $9$ of the $1$-dimensional layers and we can arrange them in $21$ different ways; therefore, the Jenga-number is $21$. On the other hand, $P_2$ is $40$, in this case.
\end{example}

\begin{table}[!t]
 \begin{center}
\begin{tabular}{|r||r|r|r|r|r|}\hline
 & $n=$2 & 3 & 4 & 5 & 6\\\hline\hline
$m=$2 & 2 & 2 & 2 & 2 & 2\\\hline
3 & 9 & 21 & 45 & 93 & 189\\\hline
4 & 272 & 49,312 & 25,485,872 &  &  \\ \hline
 \end{tabular}
 \end{center}
 \caption{$P_1(m,n)$ for small values}\label{table1}
\end{table}

\begin{table}[!t]
 \begin{center}
\begin{tabular}{|r||r|r|r|r|r|}\hline
 & $n=$2 & 3 & 4 & 5 & 6\\\hline\hline
$m=$2 & 3 & 2 & 24 & 1,344 & 1,128,960\\\hline
3 & 15 & 40& 10,123,306,543 & & \\\hline
4 & 255 & &&  &  \\ \hline
5 & 65,535 & &&  &  \\ \hline
 \end{tabular}
 \end{center}
 \caption{$P_2(m,n)$ for small values}\label{table2}
\end{table}
%
%

\begin{remark}
 For the values $P_1(2,n)$ and $P_1(3,n)$ one can easily deduce formulas. It holds
\begin{equation*}
 P_1(2,n)=2 \quad \text{ and } \quad P_1(3,n)=3(2^n-1).
\end{equation*}
Note, the second sequence also describes the number of moves to solve Hard Pagoda puzzle; further comments can be found in OEIS \cite{oeis}. Other sequences derived from these numbers are, yet, unknown to the author.
\end{remark}

\begin{proposition}
 The number of singular endomorphisms of $\overline{H(m,n)}$ is given by
 \begin{equation*}
  P_1(m,n)\cdot \#\LHC(m-1,n) \cdot (n^{m-1})!.
 \end{equation*}
\end{proposition}
\begin{proof}
 Let $\phi$ be a singular endomorphism. Then, there is a partition of $\Z_n^m$ into $1$-dimensional layers such that each part is collapsed onto a single point in the image of $\phi$. However, the image is a Latin hypercube of dimension $m-1$, class $1$ and order $n$ and consists of $n^{m-1}$ points. Thus, there are $n^{m-1}!$ choices to match the parts of the partitions with the points of its image. Conversely, this construction provides a singular endomorphism.
\end{proof}

\begin{proposition}
 The number of singular endomorphism of $\overline{H(m,n,m)}$ is given by
 \begin{equation*}
    P_2(m,n)\cdot h_{m-1}(m,n)\cdot (n^{m-1})!,
 \end{equation*}
 with $h_{m-1}(m,n)=mn$.
\end{proposition}
\begin{proof}
 Let $\phi$ be a singular endomorphism. Then, its kernel classes form a section-regular partition. Each part is a maximal coclique which is a maximal clique of $H(m,n,m)$. Moreover, the image of $\phi$ is a maximal layer and there are $(n^{m-1})!$ choices to match the parts of the partition with the points of the image.
\end{proof}
%

In fact, the number $P_2(m,n)$ is the number of semi-reduced Latin hypercubes of class $m-1$ and order $n$ as follows from the definition and Theorem \ref{theoremstructureofendomorphismsforotherhammingdistances}. So, for $P_2(m,n)$ the values in Table \ref{table2} are taken from \cite{schaeferlatinexistence}, for $n\geq 3$.


\section{The Hamming Graph for other Hamming distances}\label{section5}

Up to now, the set $S$ of distances was one of the following $\{1\},\{2,...,m\},\{m\}$ or $ \{1,...,m-1\}$. But what about other distances? In this section, we consider the set of consecutive numbers $S=\{1,...,k\}$ and prove Theorem \ref{thm4}. 

For this choice of $S$ the maximal cliques of $H(m,n,S)$ are the layers of dimension $k$, and as for $H(m,n)$, the singular endomorphisms have image a $k$-layer.
\begin{lemma}
 Let $\phi$ be a singular endomorphism of $H(m,n,S)$, for $S=\{1,...,k\}$, and let $l$ be an $s$-dimensional layer. Then, $\phi(l)$ is a layer of dimension $d$, where $k\leq d\leq s$.
\end{lemma}
\begin{proof}
 We will use induction on $m,s$ and $k$. Let $A(m,k,s)$ be the hypothesis. From the results on the Hamming graph the hypothesis $A(m,1,s)$ is always satisfied; also, $A(m,s,s)$ clearly holds for every $m$ and $s$. So, assume the hypothesis holds for $A(m,k,s)$ and show it holds for $A(m,k,s+1)$. We show that this is true by using the same argument as for $H(m,n)$.
 
 In detail, let $l$ be an $(s+1)$-layer. Then, we can split $l$ into parallel $s$-layers $l_1,...,l_n$. By induction $\phi(l_i)$ is an $s$-layer or a layer of smaller dimension, for all $i$. Now, if the dimensions of, say, $\phi(l_1)$ and $\phi(l_2)$ would differ, then there would be two maximal cliques (lines, planes, ...) $c_1$ and $c_2$ connecting $l_1$ and $l_2$ such that at least one of $\phi(c_1)$ and $\phi(c_2)$ would not be a maximal clique (line, plane, ...) in the image of $\phi$; a contradiction. Therefore, all $\phi(l_i)$ have the same dimension, say, $d$.
 
 Using the same argument, we see that each $l_i$ is collapsed to $\phi(l_i)$, and that the layers $\phi(l_i)$ must form a $(d+1)$-layer. Thus, the image $\phi(l)$ is a $(d+1)$-layer. Similarly, like in Lemma \ref{lemmaonhypercubes} we obtain uniformity.
 \end{proof}

 Consequently, we obtain the same results as for $H(m,n)$. The following statements are combined in Theorem \ref{thm4}.

\begin{corollary}
 For any singular endomorphism $\phi$ there is a maximal number $s$, such that $\phi$ maps $s$-dimensional layers to $k$-dimensional layers.
\end{corollary}


\begin{theorem}\label{theoremstructureofendomorphismsforotherhammingdistances}
 Let $S=\{1,...,k\}$. The singular endomorphisms of $H(m,n,S)$ are uniform and have rank $n^{d}$ with image a $d$-layer, for some $k \leq d \leq m-1$,
\end{theorem}
\begin{proof}
 This follows easily from the previous results.
\end{proof} 

Again it is obvious that the pre-images form Latin hypercubes of class $k$.

\begin{corollary}
 Let $S=\{1,...,k\}$. The singular endomorphisms of $H(m,n,S)$ of minimal rank are Latin hypercubes of class $k$.
\end{corollary}

Before we turn to the next section, we consider the cliques of the Hamming graph where $S=\{k+1,...,m\}$, as their maximal cliques form Latin hypercubes. In this regard we would like to remind the reader of MDS-codes. Exhaustive literature can be found this topic, but we refer to \cite[p. 71]{huffmanbook}.

In coding theory a $q$-ary code of length $n^{\ast}$, size $M^{\ast}$, and minimum distance $d^{\ast}$ is an $q$-ary $(n^{\ast},M^{\ast},d^{\ast})$ code. This code is a \emph{maximum distance separable} code (MDS-code) if it is an $q$-ary $(n,q^{k},n-k+1)$ code, where $1\leq k\leq n$.

One big question in the theory of MDS-codes is the classification of MDS-codes with regards to their parameters, meaning that we want to find all the parameters for which MDS-codes exist. This problem has been known for a long time, however a recent contribution is given by the Kokkala et. al \cite{kokkala2015classification}.

Because the vertices of $H(m,n,S)$ can be regarded as codewords, we obtain MDS-codes with the following parameters.


\begin{lemma}
 For $S=\{k+1,...,m\}$, the maximal cliques of $H(m,n,S)$ of size $n^{m-k}$ can be identified with $n$-ary $(m,n^{m-k},k+1)$ MDS-codes.
\end{lemma}
\begin{proof}
 For $S=\{k+1,...,m\}$, two vertices in $H(m,n,S)$ are adjacent if they are not in the same $k$-dimensional layer. Thus, if there is an MDS code with those parameters, then it forms a maximal clique. On the other hand, if we pick a maximal clique $c$ of this size, then each $k$-dimensional layer contains a single point of $c$. Given this, the clique has the properties of an MDS-code. 
\end{proof}

However, this result can be interpreted as a direct result on Latin hypercubes.

\begin{corollary}
 The maximal cliques of size $n^{m-k}$ of $H(m,n,S)$ are Latin hypercubes $\LHC(m-k,n)$, where $S=\{k+1,...,m\}$.
\end{corollary}
\begin{proof}
 By the preceding lemma, a maximal clique provides an MDS-code. As the codewords are of length $m$, we only pick $m-k+1$ of them and drop the remaining. This gives us a set of codewords of length $m-k+1$ where the first $m-k$ coordinates describe the position coordinates and the last coordinate as the entry coordinate of a Latin hypercube. 
\end{proof}


\section{The Hamming Graph over Cuboids}\label{section6}

The aim of this section is to generalize the results on the Hamming graphs to Hamming graphs over hypercuboids and verify Theorem \ref{thm5}. 

The graph $H(n_1,...,n_m,S)$ is a natural generalisation of $H(m,n,S)$. In particular, for $S=\{1\}$ it is the Cartesian product of (distinct) complete graphs.
\[H(n_1,...,n_m)=K_{n_1} \sq K_{n_2} \sq \cdots \sq K_{n_m},\quad \text{ for } n_1\geq n_2\geq \cdots \geq n_{m}\geq 2.\]
It is clear that these graphs also admit singular endomorphisms. In particular, if $m=2$, then $H(n_1,n_2)$ is an $n_1\times n_2$ grid which admits Latin rectangles as singular endomorphisms. Similarly, in higher dimensions Latin hypercuboids of class $1$ represent singular endomorphisms. So, the goal of this section is to describe the singular endomorphisms of $H(n_1,...,n_m)$.

\subsection{The Rectangle}

Like the square lattice graph, the rectangular lattice graph $H(n_1,n_2)$ admits only colourings with $n_1$ colours as singular endomorphisms.

\begin{lemma}
 For $n_1>n_2>1$, the singular endomorphisms of $K_{n_1} \sq K_{n_2}$ are of rank $n_1$, and they correspond to Latin rectangles. Their number is
 \begin{equation*}
  n_2\cdot \#\text{Latin rectangles}.
 \end{equation*}
\end{lemma}
\begin{proof}
Using the same arguments as for $K_n\sq K_n$, we deduce that every singular endomorphism is a colouring. The only thing we have to keep in mind is that a bigger clique cannot be mapped to a smaller clique. Thus, all its singular endomorphisms are colourings with $n_1$ colours.
\end{proof}

\subsection{The General Hypercuboid}
%
To generalize this result to higher dimensions, we need to assume that $n_1,...,n_{m-1}\geq 3$ and $n_{m}\geq 2$, for, if more than one parameter would have value $2$, there would be space for the non-uniform endomorphisms (simply collapse the diagonal vertices in the $2\times 2$ subarray). Note, that a singular endomorphism cannot map bigger cliques to smaller cliques.

\begin{lemma}\label{lemmaonhypercuboids}
 Let $\phi$ be a singular endomorphism of $H(n_1,...,n_m)$ and $l$ a $k$-layer with one of the sides of size $n_1$. Then, $\phi(l)$ is a $d$-layer, where $1\leq d\leq k$.
 Also, $\phi$ is uniformly.
\end{lemma}

\begin{proof}
 We will use induction on $m$ and $k$. For small values the hypothesis holds: $A(2,1),A(2,2)$ and $A(m,1)$. Assume $A(m,k)$ holds and show $A(m,k+1)$.
 
 Let $l$ be a $(k+1)$-subarray $n_{i_{1}}\times \cdots \times n_{i_{k+1}}$, with $n_{i_{1}}\geq \cdots \geq n_{i_{k+1}}$ and $n_{i_1}=n_1$. We split $l$ into $k+1$ part  $l_1,...,l_{n_{k+1}}$ each containing a side of length $n_1$. From here the same argument as for $H(m,n)$ shows the result.
\end{proof}

\begin{theorem}
\begin{enumerate}
 \item The singular endomorphisms of $H(n_1,...,n_m)$ are uniform of rank $n_1\cdot \prod\limits_{i\in I} n_i$, where $I$ is a proper subset of $ \{n_2,...,n_m\}$.
 \item The singular endomorphisms of rank $n_1$ are Latin hypercuboids of class $1$.
\end{enumerate}
\end{theorem}
%
%

Moreover, like for the cubic graphs when taking a set of consecutive distances $S=\{1,...,r\}$ the graphs $H(n_1,...,n_m,S)$ admit singular endomorphisms corresponding to Latin hypercubes of class $r$, and we will discuss these objects in the next chapter.

\begin{lemma}\label{lemmalatinhypercuboids}
 The singular endomorphisms of $H(n_1,...,n_m,S)$, for $S=\{1,...,k\}$, of minimal rank $n_1\cdots n_k$ are Latin hypercuboids of class $k$.
\end{lemma}

\section{Problems}

A further problem which arise from this research and possibly straightforward to state is to find all the singular endomorphisms of the graphs $H(n_1,...,n_m,S)$, for any choice of $n_i$, $i=1,...,m$, and $S$. In this regard, a concrete question is

\begin{problem}
 Given $S$, for which parameters $n_1,...,n_m$ do singular endomorphisms exist. If $S=\{1,...,r\}$, then this question involves the existence of Latin hypercuboids of class $r$.
\end{problem}

\begin{problem}
 Find a formula for the singular endomorphisms of $H(m,n,m)$. 
\end{problem}

Latin squares and Latin rectangles have been counted for decades; however, there is no publication on the numbers of Latin cuboids and Latin hypercuboids of class $r$, in general.

\begin{problem}
 Count Latin hypercuboids of class $r$.
\end{problem}

\begin{problem}
 Is there a general combinatorial interpretation for the singular endomorphisms of $H(n_1,...,n_m,S)$, for any choice of parameters?
\end{problem}

%
%

%



\begin{thebibliography}{10}

\bibitem{araujo13} J. Araújo, W. Bentz, P.J. Cameron, Groups synchronizing a transformation of non-uniform kernel, Theoretical Computer Science 498, (2013).
\bibitem{araujo15} J. Araújo, W. Bentz, P.J. Cameron, G. Royle, A. Schaefer, Primitive Groups and Synchronization, arXiv:1504.01629v2, (2015). 
\bibitem{vanbon07} J. van Bon, Finite primitive distance-transitive graphs, European Journal of Combinatorics, 28, 517–532, (2007).
\bibitem{pjclectureonsynchronization} J. Araújo, P.J. Cameron, B. Steinberg, Between primitive and 2-transitive: Synchronization and its friends, http://arxiv.org/abs/1511.03184, (2015).

\bibitem{pjc08} P.J. Cameron, P.A. Kazanidis, Cores of symmetric graphs. J. Aust. Math. Soc. 85(2),
145–154 (2008).
\bibitem{pjc13} P.J. Cameron, Dixon’s theorem and random synchronization, Discrete Mathematics 313, 1233-1236, (2013).
\bibitem{pjcdesignsbook} P.J. Cameron, J.H. van Lint, Designs, Graphs, Codes and their Links, Cambridge University Press (1991).

\bibitem{denes74} J. Dénes, A.D. Keedwell, Latin Squares and Their Applications, Academic Press, New York, (1974).
\bibitem{ethier08} J. Ethier, Strong Forms of Orthogonality for sets of Hypercubes, PhD Thesis, Pnnsylvania State University, (2008).
\bibitem{godsil11} C. Godsil, G.F. Royle, Cores of Geometric Graphs, Ann. Comb. 15, 267-276, (2011).
\bibitem{godsilbook} C. Godsil, G.F. Royle, Algebraic Graph Theory, Springer, New York, (2001).

\bibitem{hedayat99} A.S. Hedayat, N.J.A. Sloane, J. Stufken, Orthogonal Arrays - Theory and Applications, Springer New York, (1999).
\bibitem{huang14} L.-P. Huang, B. Lv, K. Wang, The endomorphism of Grassman graphs, arXiv:1404.7578v1, (2014).
\bibitem{huang15} L.-P. Huang, J.-Q. Huang, K. Zhao, On endomorphisms of alternating forms graph, Discrete Mathematics, 338, 110–121, (2015).
\bibitem{huffmanbook} W.C. Huffman, V. Pless, Fundamentals of Error-Correcting Codes, Cambridge University Press, (2003).

\bibitem{kishen50} K. Kishen, On the construction of latin and hyper-Graeco-Latin cubes and hyper-cubes, J. Indian Soc. Agricultural Stat., 2, 20-48, (1950).
\bibitem{kokkala2015classification} J.I. Kokkala, D.S. Krotov,P.R.J Ostergard, On the classification of MDS codes, IEEE Transactions on Information Theory, 61, 12, (2015).
\bibitem{mckay08} B.D. McKay, I.M. Wanless, A Census of Small Latin Hypercubes, Siam J. Discrete Math. Vol. 22, 2, pp. 719-736, (2008). 

\bibitem{moura15} L. Moura, G. L. Mullen, D. Panario, Finite field constructions of combinatorial arrays, Springer Science+Business Media New York 2015, online, (2015).


\bibitem{neumann} P.M. Neumann, Primitive Permutation Groups and their Section-Regular Partitions, Michigan Math. J.58, (2009).

\bibitem{oeis} The online encyclopedia of integer sequences, https://oeis.org/.
\bibitem{schaeferlatinexistence} A. Schaefer, Endomorphisms of the cuboidal Hamming graph, Latin hypercubes of class r, and MDS mixed codes, in preparation.
\bibitem{schaeferlatinextension} A. Schaefer, Extensions of partial Latin hypercubes of class r, in preparation.
\bibitem{schaeferlatinembeddings} A. Schaefer, Embeddings of parital Latin hypercuboids of class r, in preparation.

\end{thebibliography}
\end{document}